\documentclass[12pt]{amsart}
\usepackage{latexsym,amsmath,amsfonts,amscd,amssymb}
\usepackage{graphics}
\usepackage[all]{xy}
\usepackage{xcolor}
\newtheorem{theorem}{Theorem}[section]
\newtheorem{lemma}[theorem]{Lemma}
\newtheorem{proposition}[theorem]{Proposition}
\newtheorem{corollary}[theorem]{Corollary}

\newtheorem{definition}[theorem]{Definition}

\theoremstyle{remark}

\newcommand{\Div}{\operatorname{Div}}

\newcommand{\Ker}{\operatorname{Ker}}

\newcommand{\cA}{{\mathcal A}}

\newcommand{\cD}{{\mathcal D}}

\newcommand{\cO}{{\mathcal O}}

\newcommand{\CC}{{\mathbb C}}

\newcommand{\NN}{{\mathbb N}}

\newcommand{\QQ}{{\mathbb Q}}
\newcommand{\RR}{{\mathbb R}}

\newcommand{\ZZ}{{\mathbb Z}}

\title{A natural e\~ne product construction of the Big Witt ring}


\author[D. Barsky]{Daniel Barsky}
 \address{7 rue La Condamine, 75017 Paris, France.}
 \email{barsky.daniel@orange.fr}

\author[R. P\'{e}rez-Marco]{Ricardo P\'{e}rez-Marco}
\address{CNRS, IMJ-PRG, Universit\'e Paris Cit\'e,
B{\^a}t. Sophie Germain, Paris, France.}
\email{ricardo.perez.marco@gmail.com}

\author[J.-P. Ramis]{Jean-Pierre Ramis}
\address{Acad\'emie des Sciences, Paris and Institut de Math{\'e}matiques de Toulouse, $118$ route de Narbonne, Toulouse cedex 9, France.}
\email{ramis.jean-pierre@wanadoo.fr}

\begin{document}

\begin{abstract}
We give a  straighforward, self-contained, and natural construction of the Big Witt
ring using the e\~ne product that is defined through the action
on zeros of polynomials. This is  in contrast with classical
constructions of the Big Witt ring
using formulas out of nowhere.
\end{abstract}

\maketitle



\subsection{Introduction}

The e\~ne product of a commutative ring $A$ was defined by the second author
through the action on divisors of polynomials \cite{PM} with the
purpose of studying  its analytical properties when extended to transcendental functions
(\cite{PM2}, \cite{PM3}).
Although it is not obvious from its definition, it turns
out that the e\~ne product is a
twisted form of the multiplication in the Big Witt ring.
In this way, the definition of the e\~ne product from \cite{PM}
provides a novel and straightforward construction
of the Big Witt ring that we present in this article.
This construction shows the action as a twisted convolution
of the multiplication of the Big Witt ring
on divisors of power series representing meromorphic functions. This important fact has been surprisingly
overlooked in the classical treatments. A notable exception is Bergman \cite{Ber} Appendix B where the observation is made, but not used. It is also nearly missed in the exercises in the Commutative Algebra volumes of Bourbaki \cite{Bou} Chapter IX (exercises 42 to 46).  
Manin considers the e\~ne product, which he names ``tensor product'' and denotes by $f\otimes g$, in algebraic groups in relation with the
tensor product used in \'Etale Cohomology, but, inexplicably misses the relation with the Big Witt ring (see \cite{Ma}). In relation with this, he
presents Kurokawa's proposition for  tensoring global zeta functions, which comes from an incomplete
form of the e\~ne product \cite{Ku}. The e\~ne product
defines an e\~ne ring structure on the multiplicative group of split polynomials with constant coefficient $1$ 
and is continuous for the Krull topology. We construct a continuous extension  to
$\cA(A) = 1+XA[[X]]$ that is a twist of the Big Witt ring. The split polynomials are dense in $\cA(A)$ 
when for example $A$ is an algebraically closed field, but this is not true in general. The usual constructions
of the Big Witt ring use formulas out of nowhere. In this sense the construction presented here seem more natural. The construction is
close to the philosophy of the elegant construction by Lenstra \cite{Len}, with
simplifications coming from the use of the e\~ne product in a ring of variables.

\textbf{Structure of the article.}
In sections 2 we define the e\~ne product on finite divisors on a monoid or on a partial monoid that gives them  a ring structure. 
In section 3 we define the e\~ne product for split polynomials
over a commutative unitary ring $A$ and its relation to the e\~ne product on divisors supported on the monoid $A$ or on the 
partial monoid $A^*$ (the two approaches are presented).
In section 4 we construct the e\~ne ring structure
$(\cA(A), ., \star)$ where $\cA(A)$ is the multiplicative group of formal
power series $1+XA[[X]]$ for an arbitrary ring $A$. In section 5 we introduce Big Witt rings 
and state the main Theorem. In section 6 we identify the Big Witt ring with the twist of the e\~ne ring. 
In section 7 we single out the action of the Witt
multiplication on divisors. Section 8 is devoted to the density problem 
and we provide counterexamples. 
In section 9 we discuss the functorial properties. In section 10 we give 
a short, less than two pages long,  self-contained construction of the Big Witt ring
using the previous ideas, but avoiding the divisor discussion. We conclude in section 11 with 
a brief historical discussion of the origins of the Big Witt ring.

\subsection{E\~ne ring on finite divisors.}

Let $(G,.)$ be a magma, that is, just  a set $G$ with a binary operation.
Recall that this is a semigroup if the binary operation is associative, and
it is a monoid if the binary operation also has a neutral element.
The space of divisors $\cD(G)$
is the $\ZZ$-module of linear combinations $\delta=\sum_{g\in G} n_g .(g) $
where $n_g\in \ZZ$ is called the multiplicity (or coefficient) of $g\in G$.
The multiplication of divisors is by definition the additive structure of
this $\ZZ$-module and
$(\cD(G), .)$ is a group whose neutral element is the zero divisor ($n_g=0$ for
all $g\in G$), and \footnote{We denote multiplicatively the sum of divisors thinking about the multiplication 
of the functions that generate the divisors. This convention is also compatible with the multiplicative notation for e\~ne and
Witt rings.}
$$
\delta.\eta =\sum_{g\in G} (n_g+m_g) . (g) \ .
$$
The group of divisors $(\cD(G), .)$ is an ordered group (Bourbaki Alg\`ebre II, Chapter VI, \$ 1, 1
in \cite{Bou2}) 
and the positive cone is the monoid of positive divisors $(\cD^+(G), .)$ with non-negative  multiplicities 
$n_g\geq 0$.

The convolution or e\~ne product of divisors is defined by
\begin{equation} \label{eq:convolution}
\delta\star_G \eta =\sum_{g\in G} \left (\sum_{g_1.g_2=g} n_{g_1}m_{g_2} \right ) .(g)
\end{equation}
where the multiplicity of $g\in G$ is $0$ if there is no pair
$(g_1, g_2)\in G^2$ such that $g_1.g_2=g$.
The convolution is a binary operation for finite divisors $\cD_0(G)$, and for positive 
finite divisors $\cD_0^+(G)$. The structure $(\cD_0(G),\star_G)$ is a magma and $(\cD_0^+(G),\star_G)$
is a submagma.
The magmas $(\cD_0(G),\star_G)$ and $(\cD_0^+(G),\star_G)$  are 
associative, resp. commutative, if the binary operation of
the magma $(G, .)$ is associative, resp. commutative. 
When $(G,.)$ is a monoid, then $(\cD_0(G), \star_G)$
and $(\cD_0^+(G), \star_G)$ are monoids with neutral element the divisor 
$\delta_e =(e)$ associated to 
the neutral element $e\in G$.

\begin{proposition}[E\~ne ring structure on finite divisors on a semigroup]
\footnote{It turns out that this is the same as $\ZZ G$, the monoid ring of $G$ over $\ZZ$.} \label{prop:ene_finite_divisors}
If $(G,.)$ is a semigroup, then we have a ring structure
$(\cD_0(G), .,  \star_G)$. This is the e\~ne ring associated to the semigroup
$(G,.)$. 
The e\~ne ring is commutative if the semigroup is commutative. If $(G,.)$ is a monoid, then 
the e\~ne ring $(\cD_0(G), .,  \star_G)$ is unitary and the unit element is
the divisor $\delta_e =(e)$ associated to the neutral element $e\in G$.

If $(G,.)$ is a semigroup then we have a semiring structure on positive divisors 
$(\cD_0^+(G), .,  \star_G)$. This is the e\~ne semiring associated to the semigroup 
$(G,.)$. Also $(\cD_0(G), .,  \star_GS)$ is an ordered ring with positive cone $\cD_0^+(G)$.
\end{proposition}
We recall that a semiring (or hemiring) satisfies the same axioms of a 
ring except for the additive structure that is has a semigroup structure instead 
of a group structure (see \cite{Go}). For the definition of ordered ring see Bourbaki Alg\`ebre II, Chapter VI, 2, 1, Definition 1 in \cite{Bou2}.

\medskip

\textbf{Divisors on a partial magma.}

\medskip

A partial magma $(G,.)$ is a set with a partial binary operation defined in a subset 
$U\subset G\times G$, that is, a map 
$U\subset G\times G \to G$. The structure $(G,.)$ is a 
partial semigroup if we have associativity:  for $x,y, z \in G$ it is 
equivalent that $x.y$ and $(x.y).z$ both exist and that $y.z$ and $x.(y.z)$ both exist, and it that case we have
$$
(x.y).z = x.(y.z)
$$
and we denote by $x.y.z$ these products. The partial magma $(G,.)$ has a neutral element $e\in G$ if for any $x\in G$
we have that both $e.x$ and $x.e$ exist and 
$$
e.x=x.e=e \ .
$$

\textbf{Example.} If $A$ is a ring, then $A^*=A-\{0\}$ is a partial magma for the multiplication in $A$
defined on those pairs $(a,b)\in A^*$ such that $a.b\not=0$. The partial magma $(A^*,.)$ is a magma 
if and only if $A$ has no divisors of $0$.

\medskip

The construction of the e\~ne ring of divisors of a magma extends to the case when $(G,.)$ is a partial magma.
The definition is the same except for the convolution formula (\ref{eq:convolution}) where the summation only 
extends to those $(g_1, g_2)\in G^2$ such that $g_1.g_2$ exists (and the empty sum is $0$). We have 
that Proposition \ref{prop:ene_finite_divisors} extends to partial semigroups.

\begin{proposition}[E\~ne ring structure on finite divisors on a partial semigroup] \label{prop:ene_finite_divisors_partial}
If $(G,.)$ is a partial semigroup, then we have a ring structure
$(\cD_0(G), .,  \star_G)$. This is the e\~ne ring associated to the partial semigroup
$(G,.)$. 
The e\~ne ring is commutative if the partial semigroup is commutative. If $(G,.)$ is 
a partial monoid, then 
the e\~ne ring $(\cD_0(G), .,  \star_G)$ is unitary and the unit element is
the divisor $\delta_e =(e)$ associated to the neutral element $e\in G$.

If $(G,.)$ is a partial semigroup then we have a semiring structure on positive divisors 
$(\cD_0^+(G), .,  \star_G)$. This is the e\~ne semiring associated to the semigroup 
$(G,.)$. Also $(\cD_0(G), .,  \star_G)$ is an ordered ring with positive cone $\cD_0^+(G)$.
\end{proposition}

\subsection{E\~ne ring structure associated to a commutative ring.}

We consider an unitary commutative ring $A$ and the 
multiplicative group of formal power series $\cA(A)=1+XA[[X]]$.
We consider also the commutative multiplicative
semigroup $P_A\subset 1+XA[X]\subset \cA(A)$ of \textit{split polynomials}
with constant coefficient $1$. So for $f\in P_A$ we can write
$$
f(X)=\prod_{a\in A} (1-aX)^{n_a}
$$
with $n_a\in \NN$, and almost all $n_a=0$. 
Note that the split factorization is never unique as $n_0$ can be arbitrary and we can always 
multiply by a positive integer power of the 
constant polynomial equal to $1$. Of course, 
the non-uniqueness can be worse when $A$ has zero divisors. 
Let $S_A$ be the 
multiplicative subgroup $S_A \subset \cA(A)$ generated by $P_A$. We name $S_A$
the space of \textit{split ``rational'' functions}. We can write 
any element $f(X)\in S_A$ as a finite product
$$
f(X)=\prod_{a\in A} (1-aX)^{n_a}
$$
with $n_a\in \ZZ$, and almost all $n_a=0$. Note that 
$$
(1-aX)^{-1}=1+\sum_{k\geq 2} a^k X^k \in \cA(A) \ .
$$

We can consider divisors with support on $A^*$ considering the partial monoid structure 
$(A^*,.)$ (the convention is that $a.b$ does not exist in $A^*$ when $a.b=0$ in $A$),
or divisors with support  on $A$ considering the monoid structure $(A,.)$.
We start considering the  first case that geometrically is more natural, but maybe algebraically 
less orthodox.

\medskip

\textbf{Divisors with support on the partial monoid $(A^*,.)$.}

\medskip

The following Proposition is obvious from the definition of $P_A$.

\begin{proposition}[Positive divisor map]
The positive divisor map $\varphi_{A^*}^+: \cD_0^+(A^*) \to P_A\subset A[X]$ defined by 
$$
\delta =\sum_{a\in A^*} n_a . (a) \mapsto f(X)=\prod_{a\in A^*} (1-aX)^{n_a}
$$
is a surjective semigroup morphism. 
\end{proposition}

\begin{lemma}
If $A$ has no zero divisors then  $\varphi_{A^*}^+$ is an isomorphism of semigroups, $\Ker \varphi_{A^*}^+ =\{0\}$.
\end{lemma}
\begin{proof}
Consider  $\delta \in \Ker \varphi_{A^*}^+$, with $\delta =\sum_{k=1}^m n_{a_k} . (a_k)$ given 
by a finite formal sum of $m\geq 0$ non-zero terms, $a_k\not=0$. If $m\geq 1$ , then looking at 
the coefficient of $X^n$ we have
$a_1\ldots a_m=0$. Since $A$ has no zero divisors, we have for some $1\leq k_0 \leq m$, $a_{k_0}=0$. Therefore $m=0$ and 
$\Ker \varphi_{A^*}^+ =\{ 0\}$.
\end{proof}

In general when $A$ has divisors of $0$,  $\varphi_{A^*}^+$ is not an isomorphism
but we can define the e\~ne product on $P_A$ by the following procedure:

\begin{theorem}[E\~ne product on $P_A$] \label{thm:well_defined}
Given $f, g\in P_A$, we choose $\delta_f, \delta_g \in  \cD_0^+(A^*)$, then we 
define 
$$
f\star g = \varphi_A^+(\delta_f\star_{A^*} \delta_g)
$$
that is,
\begin{equation}\label{eq:formula1}
f\star g = \prod_{\exists a.b\in A^*} (1-abX)^{n_a m_b}=\prod_{c=a.b\in A^*} (1-cX)^{\left (\sum_{a.b=c} n_a.m_b \right )}
\end{equation}
and the result does not depend on the choices of $\delta_f$ and  $\delta_g$.
\end{theorem}

\begin{proof}
Observe that the coefficients $(A_k)_{k\geq 1}$ of the expansion of the product
$$
f(X)=\prod_{a\in A} (1-aX)^{n_a}=1+A_1 X+A_2X^2+\ldots
$$
are the elementary symmetric functions evaluated in the $a$'s repeated with multiplicities, that we also name 
elementary symmetric functions of the associated divisor.
If we choose another divisors $\delta'_f$ that incarnates $f$  then the elementary symmetric 
functions of $\delta_f$ and  $\delta'_f$ give the same coefficients $(A_k)$ that  coincide. 
From  Lemma \ref{lem:Fundamental} in Section 3 we have that  the coefficients of
$f\star g$ are universal polynomial functions with integer coefficients of the elementary symmetric functions of 
$\delta_f$ and $\delta_g$, hence they are universal polynomial formulas with 
integer coefficients on the 
coefficients of $f$ and $g$. Therefore, the result 
is independent of the choices of $\delta_f$ and $\delta_g$. 
\end{proof}

\medskip

\textbf{Divisors with support on the monoid $(A,.)$.}

\medskip
The following Proposition is inmediate from the definition of $S_A$.

\begin{proposition}[Divisor map] 
The divisor map $\varphi_A: \cD_0(A) \to S_A$ defined by 
$$
\delta =\sum_{a\in A} n_a . (a) \mapsto f(X)=\prod_{a\in A} (1-aX)^{n_a}
$$
is a surjective group morphism. 

Its restriction to positive divisors $\varphi_A^+: \cD_0^+(A) \to P_A$ is a 
surjective morphism of semigroups.
\end{proposition}

Observe that, contrary to the precedent situation, even when $A$ has no divisors of $0$,  
we have that $\varphi_A^+$, resp. $\varphi_A$, is never an isomorphism since  
$\Ker \varphi_A^+$, resp.
$\Ker \varphi_A$,
always contains  $\NN .(0)$, resp. $\ZZ.(0)$. Note that $\ZZ.(0)=\bigl( (0) \bigl)$ is a principal 
ideal of the e\~ne ring
$(\cD_0(A),.,\star_A)$.

Again, we desire to define the e\~ne product $f\star g$ for $f,g\in P_A$, or for $f,g\in S_A$, 
by 
\begin{equation} \label{eq:ene_formula}
(f\star g)(X) = \prod_{a,b \in A} (1-ab X)^{n_a.m_b} = \prod_{c=a.b\in A^*} (1-cX)^{\left (\sum_{a.b=c} n_a.m_b \right )}
\end{equation}
but we need to justify that the result is independent of the choice of the 
presentation of $f(X)$ and $g(X)$ as split factorizations. When $A$ has no divisors of $0$ 
the same argument as before proves that $\Ker \varphi_A^+ = \NN .(0)$ and 
$\Ker \varphi_A = \ZZ .(0)$. Therefore, when $A$ has no divisors of $0$ 
the split factorizations of $f$ is  unique up to the freedom of choosing $n_0$ that does not change the product. 

For a general ring $A$, the independence of the definition of the e\~ne product 
of the choice of the divisors incarnating 
the functions in $P_A$ (Theorem \ref{thm:well_defined}) works as well, 
but not for $S_A$. Instead 
we can use the following  alternative 
argument due to Hendrik Lenstra if we insist in defining the e\~ne product for the 
full group $S_A$ (something that is not necessary for the construction in the next section, 
but is certainly satisfactory to have). It is noteworthy to observe that 
the argument avoids 
the use of elementary symmetric functions.

\begin{proposition}[H. Lenstra]
 The kernel  $\Ker \varphi_A$ is an ideal of the e\~ne ring $(\cD_0(A), ., \star)$, and the group isomorphism 
 $$
 \cD_0(A)/\Ker \varphi_A \approx S_A
 $$
 induces on $S_A$ a ring structure and formula (\ref{eq:ene_formula}) for 
 the e\~ne product is well 
 defined independently of the split factorizations.
\end{proposition}

\begin{proof}
In order to prove that $\Ker \varphi_A$ is an ideal of $\cD_0(A)$, it is enough to prove 
that for $c\in A$ we have $(c)\star \Ker \varphi_A \subset \Ker \varphi_A$. 
The ring isomorphism $\psi_c : A[[X]]\to A[[X]]$ (for the usual ring 
structure of $A[[X]]$) defined by 
$$
\psi_c \left (\sum_{n\geq 0} a_n X^n \right ) = \sum_{n\geq 0} a_n (cX)^n 
$$
induces on the multiplicative group $S_A$ a group morphism $\psi_c : S_A \to S_A$ (we keep the same notation for this restriction). For a divisor $\delta\in \cD_0(A)$, 
$$
\delta = \sum_{a\in A} n_a.(a)
$$
we have
$$
\prod_{a\in A} (1-ca X)^{n_a} = \prod_{a\in A} (1-a (cX))^{n_a}
$$
which means
$$
\varphi_A((c)\star \delta)=\psi_c(\varphi_A(\delta)) \ .
$$
Therefore, if $\delta\in \Ker \varphi_A$ we get $(c)\star \delta \in \Ker \varphi_A$.

Finally, observe that the choice of a split factorization corresponds 
to the choice of $\delta $ modulo $\Ker \varphi_A$.
\end{proof}

\textbf{Remark.} As observed before we always have $\ZZ .(0) \subset \Ker \varphi_A$. When 
$A$ has no divisors of $0$ then we have $\ZZ .(0) = \Ker \varphi_A$, and this gives the isomorphism of e\~ne rings,
$$
 \cD_0(A)/\ZZ .(0) \approx \cD_0(A^*)  \ .
$$

In this case, the map $\varphi_{A^*}: \cD_0(A^*) \to S_A$ is an isomorphism of e\~ne rings, and 
it induces an isomorphism of e\~ne semirings $\varphi^+_{A^*}: \cD_0^+(A^*) \to P_A$.

\medskip

\begin{definition}
We define the twisted divisor of $f\in S_A$, as
$$
\#\Div(f(X)) =\sum_{a\in A} n_a . (a) \in \cD_0(A)/ \Ker \varphi_A \  .
$$ 
\end{definition}

This is just $\Div f(1/X)$ when this makes sense.

\begin{corollary}(E\~ne product) \label{cor:def_ene}
 For $f,g \in S_A$, the e\~ne product $f\star g \in S_A$ is uniquely defined by
$$
\#\Div (f\star g) = \#\Div (f) \star_{A} \#\Div(g)
$$
\end{corollary}

\subsection{Continuous extension of the e\~ne product to $\cA(A)$.} \label{sect:extension}
We consider an arbitrary unitary commutative ring $A$.
Let $\cA(A)=1+XA[[X]]$ the multiplicative group of formal power series with coefficients in $A$ endowed 
with the Krull topology where a bases of neighborhoods of $0$ is given by $(1+X^nA[[X]])_{n\geq 1}$.
For the next Lemma we  
consider two polynomials given in split form.
\begin{lemma}\label{lem:Fundamental}
Let $P(x),Q(x) \in  P_A \cap A[X]$ be two split polynomials and write
\begin{align*}
P(X) &=1+a_1 X+a_2 X^2+\ldots =\prod_{a\in A} \left (1-aX\right )^{n_a}\\
Q(X) &=1+b_1 X+b_2 X^2+\ldots =\prod_{b\in A} \left (1-bX \right )^{m_b}
\end{align*}
with all but a finite number $(n_a), (m_b), (a_n), (b_n)$ being $0$.
We have that $P\star Q \in P_A$ is a split polynomial and
$$
(P\star Q)(X) = 1+c_1 X+c_2 X^2+\ldots =\prod_{c\in A} (1-cX)^{\left (\sum_{a.b=c} n_a.m_b \right )}
$$
with
$$
c_n=C_n(a_1, a_2,\ldots, a_n, b_1, b_2, \ldots, b_n)
$$
where the $C_n\in \ZZ[X_1, \ldots, X_n, Y_1,\ldots, Y_n]$, for $n\geq 1$,
are universal polynomial with integer coefficients.
\end{lemma}

\begin{proof} In the ring $\ZZ[(X_k)_{1\leq k\leq n}, (Y_k)_{1\leq k\leq n}, T]$ of polynomials in $2n+1$ variables consider the expansions
$$
\prod_{k=1}^{n} (1-X_k T) = \sum_{k=1}^{n} \Sigma_k(X) T^k \ \ \text{and} 
\ \ \prod_{l=1}^{n} (1-Y_l T) = \sum_{l=1}^{n} \Sigma_l(Y) T^l 
$$
where the $\Sigma_k(X)$, resp. $\Sigma_l(Y)$,  are the elementary symmetric functions on the 
variables $(X_k)_{1\leq k\leq n}$, resp. $(Y_k)_{1\leq k\leq n}$.
Then we have
$$
\prod_{k,l=1}^n (1-X_k Y_l T) = \sum_{m=1}^{n^2} \Sigma_{m,n} (X,Y) T^m
$$
The polynomial $\Sigma_{m,n}(X,Y)\in \ZZ[(X_k)_{1\leq k\leq n},(Y_k)_{1\leq k\leq n}]$ is symmetric individually
on the two groups of variables $(X_k)_{1\leq k\leq n}$ and $(Y_k)_{1\leq k\leq n}$. Applying the Fundamental Theorem
on Symmetric Functions with the base ring $\ZZ[(X_k)_{1\leq k\leq n}]$ (see Bourbaki
Algebra Chapter 4 \cite{Bou}),  we have that
$$
\Sigma_{m,n} (X,Y) \in \ZZ[(X_k)_{1\leq k\leq n}] [(\Sigma_l((Y_k)_{1\leq k\leq n})_{1\leq k
\leq n}]
$$
Applying a second time the same Theorem to each coefficient
in $\ZZ[(X_k)_{1\leq k\leq n}]$ that are symmetric in the variables $(X_k)_{1\leq k\leq n}$, we obtain
that
$$
\Sigma_{m, n}(X,Y) \in \ZZ[(\Sigma_k(X))_{1\leq k\leq n}, (\Sigma_k (Y))_{1\leq k\leq n}]
$$
For $1\leq m\leq n$ the polynomials $\Sigma_{m, n}(X,Y) =\Sigma_{m}(X,Y)$ stabilize and are independent 
of $n\geq m$ because the $m$-th germ of $P\star Q$ depends only on the $m$-th germ of $P$ and $Q$.
Replacing the variables by the values in $A$, we have $c_n=\Sigma_n((a_l)_{1\leq l\leq n},(b_l)_{1\leq l\leq n})$ and the result follows.

\end{proof}

\begin{corollary}[Universality] \label{cor:universality}
For an arbitrary ring $A$ we define the e\~ne ring structure $(\cA(A),.,\star)$ using 
the coefficient formulas given by the universal polynomials $(C_n)_{n\geq 1}$. 
This defines  a commutative topological ring  for the Krull topology.

The e\~ne ring structure $(\cA(A),.,\star)$ is a continuous extension 
of the  e\~ne semi-ring structure on $(P_A,.,\star)$.
\end{corollary}
\begin{proof}
Note that it makes sense to define coefficientwise the e\~ne product $f\star g$ for any $f,g\in  \cA(A)$ because
the polynomials $(C_n)_{n\geq 1}$ have integer coefficients. 
This binary operation
is continuous for the Krull topology since
if $f_1(X)-f_2(X)\in X^nA[[X]]$ and $g_1(X)-g_2(X)  \in X^nA[[X]]$, then from the
universal formulas we have
$(f_1\star g_1)(X)-(f_2\star g_2)(X)  \in X^nA[[X]]$.
The extension of the e\~ne product defined in this way extends the commutative
semi-ring structure $(P_A,.,\star)$ to a commutative ring structure since the properties as  associativity or
commutativity are encoded in universal polynomial relations on the coefficients
that continue to hold, and also the coefficients of the inverse of an element $f$ in $\cA$ are universal polynomials
with integer coefficients on the coefficients of $f$. 
\end{proof}

\begin{corollary}
When  $P_A$ is dense in $\cA(A)$ the extension of the e\~ne product to $\cA(A)$
is unique.
\end{corollary}
For example, this happens when $A$ is a algebraic closed field or when $A$ is a
ring where every $a\in A$ has a $n$-th root and the polynomial $1-X^n$ splits
for every $n\geq 2$. In general this is not true. We prove in section \ref{counterexample} 
that neither $P_\RR$ is dense in $S_\RR$, nor $S_\RR$ is dense in $\cA(\RR)$.

The next Corollary follows from the universality of the formulas but not from 
the density.

\begin{corollary}
The e\~ne product defined in $\cA(A)$ in Corollary \ref{cor:universality} coincides with the e\~ne product 
in $S_A$ defined by equation (\ref{eq:ene_formula}).
\end{corollary}

\begin{proof}
By construction it coincides in $P_A$. When $A=\CC$, simply looking at divisors,  we have
\begin{align*}
f^{-1}\star g^{-1} &=f\star g \ , \\
f^{-1}\star g &=(f\star g)^{-1} \ .
\end{align*}
By universality, these equations remain true for any ring $A$. More precisely,
the polynomial equations that give these identities coefficientwise are the same 
as those for $A=\CC$. When we replace in this formulas $f$ and $g$ by $1+aX$ and $1+bX$, 
we can see that the extension of the e\~ne product in $P_A$ to $\cA(A)$ coincides on the 
generators of $S_A$ with the one defined by equation (\ref{eq:ene_formula}) (or Corollary \ref{cor:def_ene}), hence it coincides
in all $S_A$.
\end{proof}

\subsection{The Big Witt ring.}

We consider the multiplicative group of formal power series
$\cA(A) = 1+XA[[X]]$ and the Big Witt ring structure $(\cA(A),.,\star_w)$
defined by G. Bergman (\cite{Ber}, 1966), following Witt work on p-adic rings
(\cite{Witt}, 1937). A very direct
construction was given by P. Cartier
(\cite{Ca}, 1967), and another very enlighting construction by H. Lenstra (\cite{Len}, 2002). We endow $\cA(A)$ with the Krull topology generated
by the ideal $(X)$ and the Big Witt ring is a topologcal ring.

\begin{theorem}
The twisted e\~ne ring structure $(\cA(A),.,\check \star)$ defined by
$$
f\check\star g =(f\star g)^{-1}
$$
coincides with the Big Witt ring.
\end{theorem}

We prove this Main Theorem in the next section.


\subsection{Identification with the twisted Big Witt ring.}\label{sect:identification}

We work in the ring of formal power series in a countable set of variables with
rational coefficients $\QQ[[(X_k)_{k\geq 1}, (Y_k)_{k\geq 1}, \ldots ]]$. For background,
we refer the reader to Bourbaki, Alg\`ebre, Chapter IV, section 4, \cite{Bou}.
This ring has characterictic $0$.

\begin{lemma}[Formal Exponential form]
\label{lemmaexp}
We have
{\small
$$
f(T)= \prod_{k=1}^{+\infty} (1-X_k T) =\exp\left (- \sum_{n=1}^{+\infty} \frac1n N_n (X) T^n \right ) \ \ \text{with} \ \  N_n(f)=N_n (X)=\sum_{k=1}^{+\infty} X_k^n
$$
}
and
{\small
$$
f(T)=\prod_{k=1}^{+\infty} (1-X_k T^k)^{-1} =\exp\left (\sum_{n=1}^{+\infty} \frac1n W_n(X) T^n \right ) \ \ \text{with} \ \  W_n(f)=W_n(X) = \sum_{d|n} d X_d^{n/d} \ .
$$
}
\end{lemma}

\begin{proof}
Develop $\log (1-X_kT)$ or $-\log (1-X_kT^k)$ in the exponential and regroup terms with the same exponent of $T$. The first case is clear. For the second case we have
\[
-\sum_{k\geq 1} \log (1-X_k T^k)= \sum _{k,l\geq 1}   \frac1l X_k^l \, T^{kl}
=\sum _{n\geq 1} \frac1n\left(\sum_{k,l\geq 1; kl=n}   k X_k^l\right) T^n = \sum_{n\geq 1} \frac1n W_n(X) T^n .
\]
\end{proof}

The first left formula in the Lemma is the exponential form of the e\~ne product that appears in \cite{PM} Section 4. The second left formula is the exponential
form for the Big Witt product that appears in \cite{Witt1} formula (1.3).

The power series  $N_n(X) \in \ZZ[[(X_k)_{k\geq 1}]]$ are the classical Newton sums of the variables $(X)$. These are power series in an infinite number of variables.
The polynomials $W_n(X) = \sum_{d|n} d X_d^{n/d} \in \ZZ[X_1, \ldots , X_n]$, for all $n\geq 1$,
also called ``ghost components'', appear first in print independently in Bergman \cite{Ber} p.180 and in Lang's Algebra book as an exercise \cite{La} (Exercise in Chapter VIII). Lang gives credit to an oral communication by Witt, and Witt left a manuscript
note from a seminar he gave in Hamburg in June 1964 (see \cite{Witt2} p.164, and  also
the essay by Harder \cite{Har}). It seems historically justified to name
them the Bergman-Witt polynomials. Traditionally the pre-1940 theory of Witt rings
is derived from the subsequence $W_{p^n}$ (see \cite{Ber}, \cite{Bou},  \cite{Ca}, \cite{La}, \cite{Se}). The Big Witt multiplication is defined by
$$
W_n(f\star_w g) = W_n(f).W_n(g) \ .
$$
So the Witt multiplication corresponds to the simple multiplication of ghost components.
They are introduced in the construction
of the Big Witt ring without proper motivation. A  notable exception is Lenstra's elegant
construction \cite{Len} that does not use them.
In our construction, the Bergman-Witt polynomials appear naturally
through the exponential form of the e\~ne product, as shown in the
 above Lemma \ref{lemmaexp}, and the next Lemma gives the exponential form of the
 e\~ne product.

\begin{lemma}[Formal exponential form of the e\~ne product]
{\small
$$
\left (\prod_{k=1}^{+\infty} (1-X_k T)\right ) \star \left (\prod_{l=1}^{+\infty} (1-Y_l T)\right )
=\exp \left (- \sum_{n=1}^{+\infty} \frac1n (N_n(X). N_n(Y)) T^n \right) \ .
$$
}
\end{lemma}

Hence, the e\~ne corresponds to the simple multiplication of the Newton sums.

\begin{proof}
We have
{\small
$$
\left (\sum_{k=1}^{+\infty} X_k^n \right ).\left (\sum_{l=1}^{+\infty} Y_l^n \right )
= \left (\sum_{k, l=1}^{+\infty} (X_kY_l)^n \right ) \ .
$$
}
\end{proof}

\begin{corollary}[Relation with the Hadamard product]
If
$$
\cD f(T) =-T \, \frac{f'(T)}{f(T)}
$$
and if we denote by $\odot$ the Hadamard product, then we have
 \begin{align*}
\cD(f\star g) &= \cD f(T)\odot \cD g(T)   \\
\cD(f\star_w g) &= -\cD f(T)\odot \cD g(T)   \\
 \end{align*}
\end{corollary}

The first identity appears in \cite{PM}, Theorem 10.5, the second
one in \cite{Len}, page 1241.

\begin{corollary}
We have $\check \star =\star_w$
\end{corollary}
\begin{proof}
Observe that $\cD (f(T)^{-1}) = -\cD f(T)$.
\end{proof}

\begin{corollary}
The result holds for an arbitrary commutative ring $A$.
\end{corollary}
\begin{proof}
We have universal polynomials $(Q_n^w)$ with
{\small
$Q_n^w\in \ZZ[Z_1,\ldots , Z_n]$ such that
$$
f\star_w g = 1+\sum_{k\geq 1} Q_k^w \, T^k = \left ( 1+\sum_{k\geq 1} Q_k \, T^k\right )^{-1}  =(f\star g)^{-1}\ .
$$
}
We can replace variables $(X_k)_{ k\geq 1}$ and $(Y_k)_{ k\geq 1}$ with arbitrary
values in the ring $A$, including the case when $A$ has non-zero characteristic since the polynomials have integer coefficients. The identity remains true for all commutative rings.
\end{proof}

\subsection{Action on divisors of the Big Witt multiplication.}

The following Corollary is obvious from our construction. We single it out
because it is generally overlooked in the standard references (except in
\cite{Ber} Appendix 8). We define the split ``rational'' functions $S_A$ as those 
elements in $\cA(A)$ that can be writen as finite products
$$
f(X)=\prod_{a\in A} (1-aX)^{n_a}
$$
with $n_a\in \ZZ$. Note that $(1-aX)^{-1} =1 +aX+a^2X^2+\ldots \in \cA(A)$.

\begin{corollary}
The multiplicative sub-group  $S_A\subset \cA(A)$ is invariant by the Big Witt ring multiplication.
\end{corollary}

This is to be compared with the invariance of $P_A\subset A[X]$ by the e\~ne product.

\subsection{Density of $P_A$ and $S_A$ and counterexamples.} \label{counterexample}

\begin{proposition}
Let $A$ be a commutative ring such that for every $n\geq 2$, every $a\in A$
has an $n$-th root, and $1-X^n \in P_A$. Then, $P_A$ and $S_A$ are dense in $\cA(A)$.
\end{proposition}

\begin{proof}
From \cite{Ca} (or \cite{La}, exercise in chapter VIII) we have that
any element $f\in \cA(A)$ can be written as an
infinite product
$$
f(X)=\prod_{n=1}^{+\infty} (1-a_nX^n)^{-1}
$$
Also, considering the inverse of $f$, any $f\in \cA(A)$, can be written as
$$
f(X)=\prod_{n=1}^{+\infty} (1-a_nX^n)
$$
If $A$ satisfies the conditions, then for every $n\geq 1$,
each factor $1-aT^n$ splits as
$$
1-aX^n=1-(bX)^n=\prod_{\omega^n=1} (1-b\omega X)
$$
where $b^n=a$. This proves that finite products that build $P_A$ and $S_A$ are dense in $\cA(A)$.
\end{proof}

\begin{corollary}
 When $A$ is an algebraically closed field then $P_A$ and $S_A$ are dense in $\cA(A)$.
\end{corollary}

We have counter-examples for non-algebraically closed fields.

\begin{proposition}
The group $S_\RR$ is not dense in $\cA(\RR)$.
\end{proposition}

\begin{proof}
Consider $A=\RR$. We prove that  $1+X^2 \in \cA(\RR)$ is not in the closure of
$S_\RR $. Otherwise  there are finite sequences $(a_k)_{1\leq k\leq n}$, $a_k\in \RR$, and  $(\epsilon_k)_{1\leq k\leq n}$, $\epsilon_k=\pm 1$, such that
$$
\prod_{k=1}^n (1+a_k X)^{\epsilon_k} = 1+X^2 +\cO(X^3)
$$
then from  the coefficients of $X$ and $X^2$ we have
$$
\sum_{k} \epsilon_k a_k =0 \ \ \text{and} \ \ \sum_{k\not= l} \epsilon_k \epsilon_l a_k a_l + \sum_{\epsilon_k=-1}  a_k^2 =1 \ .
$$
The second Newton relation gives a contradiction
$$
0=\left ( \sum_{k} \epsilon_k a_k \right )^2 = \sum_{k} a_k^2 + \sum_{k\not= l} \epsilon_k \epsilon_l a_k a_l= \sum_{k} a_k^2 + \left ( 1- \sum_{\epsilon_k=-1} a_k^2 \right )
= 1+\sum_{\epsilon_k=1} a_k^2 \geq 1 \ .
$$
\end{proof}

Also, in general we don't have that $P_A$ is dense in $S_A$.

\begin{proposition}
 The semigroup $P_\RR$ is not dense in the group $S_\RR$.
\end{proposition}
\begin{proof}
Let $a\in \RR$ with $a\not=0$. We prove that $(1+aX)^{-1}$ is not in the closure of $P_\RR$. We have
$$
(1+aX)^{-1} = 1-aX+a^2 X^2 +\ldots
$$
By contradiction there would be a sequence $(a_k)_{1\leq k\leq n}$, $a_k\in \RR$, such that
$$
\prod_{k=1}^n (1+a_k X) = 1-aX+a^2X^2 +\cO(X^3).
$$
Then we have
$$
\sum_{k=1}^n a_k =-a \ \ \text{and} \ \ \sum_{k\not= l} a_k a_l =a^2 \ .
$$
so
$$
a^2=\left (\sum_{k=1}^n a_k \right )^2 =\sum_{k=1}^n a_k^2 + \sum_{k\not= l} a_k a_l = \sum_{k=1}^n a_k^2 + a^2
$$
therefore
$$
\sum_{k=1}^n a_k^2 =0
$$
which implies $a_k=0$ for all $k\geq 1$ and $a=0$. Contradiction.

\end{proof}

\subsection{Functoriality of the e\~ne product.}

We prove the functoriality property given by Lenstra \cite{Len} for the
Big Witt ring structure and the e\~ne ring structure
(we give the statement in this last situation).
As noted by Lenstra,
$\cA$ is a functor from the  category of rings into the category
of abelian groups. We prove that it is also a functor of rings
with $\cA(A)$ endowed with the e\~ne ring structure.

\begin{theorem}
A morphism $\varphi: A\to B$  of commutative rings induces a commutative diagram
\[
\xymatrix{
 \mathcal{A}(A) \times \mathcal{A}(A) \ar@<0ex> [r]^{\quad \, \star}
 \ar@<0ex>[d]_{(\mathcal{A}(\varphi),\mathcal{A}(\varphi))} &\mathcal{A}(A)
 \ar@<0ex>[d]^{\mathcal{A}(\varphi)} \\
  \mathcal{A}(B) \times \mathcal{A}(B)   \ar@<0ex>[r]^{\quad \, \, \star}     & \mathcal{A}(B)
}
\]
\end{theorem}

\begin{proof}
We consider the enveloping ring of polynomials  $\ZZ[(X_a)_{a\in A}]$ with an infinite, eventually uncountable, number of variables. 
We refer to Bourbaki's Algebra Chapter IV section 1 \cite{Bou} for calculus on formal power 
series in an infinite number of variables.
There is a surjective quotient
$$
\pi_A : \ZZ[(X_a)_{a\in A}] \to A
$$
such that $\pi_A(X_a)=a$ where the kernel encodes all the ring axioms and the ring relations.
The map $\varphi: A\to B$ induces a map $(X_a)_{a\in A} \rightarrow (X_b)_{b\in B}$ (defined by $X_a \mapsto X_{\varphi(a)}$) and a ring morphism $\ZZ[(X_a)_{a\in A}] \rightarrow \ZZ[(X_b)_{b\in B}]$. We have a commuting diagram
\[
\xymatrix{
 \ZZ[(X_a)_{a\in A}] \ar@<0ex> [r]
 \ar@<0ex>[d]_{\pi_A} & \ZZ[(X_b)_{b\in B}]
 \ar@<0ex>[d]^{\pi_B} \\
   A  \ar@<0ex>[r]^{\varphi}     & B
}
\]
The quotient also defines a ring morphism between e\~ne
rings $(\cA(\ZZ[(X_a)_{a\in A}]), .,\star)$ and $(\cA(A), ., \star)$. The result follows.

\end{proof}

\textbf{Remark.} In the proof of the properties of
the e\~ne product in section \ref{sect:identification}, we use
the ring of formal power series in a countable set of variables $\QQ[[(X_k)_{k\geq 1}, (Y_k)_{k\geq 1}, \ldots ]]$.
In this section, for different purposes,
we consider the enveloping ring of polynomials
$\ZZ[(X_a)_{a\in A}]$ which is generated  a priori
by an uncountable number of variables.
In the next section we consider a  ring of formal power series with an a priori
uncountable number of variables. One should not be confused with the instrumental countable variables that
serve to construct the e\~ne product and to prove its properties, and those other
variables that are used
to construct the enveloping ring that serve for establishing functorial properties.
Rings of variables are useful in proving the properties of the e\~ne product because they are of characteristic zero (hence, for example, we
can write directly exponential forms without needing a yoga
with logarithmic derivatives). Note that if the ring $A$ has non-zero characteristic it cannot be embedded into a larger ring of characteristic zero. This is the  reason
why is natural to consider
enveloping rings. Their usefulness can be seen, for example,
at the begining of section 9.10 in \cite{Ha} where there is some confusion and the ``larger ring''
considered should instead be  an enveloping ring.


\subsection{Short self-contained proof.}

Consider the ring
$$
A=\QQ[[(X_k)_{k\geq 1}, (Y_l)_{l\geq 1},(Z_m)_{m\geq 1},\ldots ]]
$$
of commuting variables\footnote{Maybe this time an uncountable number, but we write them in finite or
countable groups.},
the multiplicative group $\cA(A) =1+TA[[T]]$ and the subgroup
$S[T]\subset A[T]$ generated by $(1-XT)_{X}$, where the $X$ run over monomials of $A$, i.e.
$f\in S[T]$ if there is a finite or countable sequence $(X_k)_{k\geq 1}$ of monomials in $A$ such that
{\small
$$
{f(T) =\prod_{k=1}^{+\infty} (1-X_k T) = \sum_{k\geq 0} \Sigma_k(X)T^k}
$$}
where $\Sigma_k(X)$ is the $k$-th elementary symmetric function on the monomials
$(X_k)_{k\geq 1}$ ($\Sigma_0(X)=1$).
\begin{definition}
Let $f,g \in S[T]$,
we define the e\~ne product $f\star g\in S[T]$ by
{\small
$$
f\star g =\left (\prod_{k=1}^{+\infty} (1-X_k T)\right ) \star \left (\prod_{l=1}^{+\infty} (1-Y_l T)\right ) =\prod_{k,l\geq 1} (1-X_k Y_l T)
=\sum_{k\geq 0}\Sigma_k(X,Y) T^k
$$
}
\end{definition}
\begin{lemma}
$(S(T),.,\star)$ is a commutative ring.
\end{lemma}
\begin{proof}
Obvious. If $X_k$ and $Y_l$ are monomials, then $X_k Y_l$ is a monomial.
\end{proof}
\begin{lemma}For $n\geq 0$,
$$\Sigma_n(X,Y)=Q_n(\Sigma_1(X),\Sigma_2(X), \ldots ,\Sigma_n(X), \Sigma_1(Y),\Sigma_2(Y), \ldots ,\Sigma_n(Y))
$$
for a universal polynomial $Q_n\in \ZZ[Z_1,\ldots , Z_n]$.
\end{lemma}
\begin{proof} The polynomial $\Sigma_n(X,Y)$ is symmetric individually
on each group of variables $(X_k)_{1\leq k\leq n}$ and $(Y_k)_{1\leq k\leq n}$. Using the Fundamental Theorem on symmetric functions (FTSF) over the coefficient ring
$\ZZ[(X_k)_{1\leq k\leq n}]$ we have that   $\Sigma_n(X,Y)$ is a polynomial
with coefficients in $\ZZ[(X_k)_{1\leq k\leq n}]$ of  $\Sigma_1(Y),\Sigma_2(Y), \ldots ,\Sigma_n(Y)$. Applying a second time to each coefficient the  FTSF over the coefficient ring $\ZZ$ we prove the result.
\end{proof}

\begin{corollary}
 The e\~ne product extends as a binary operation to $\cA(A)$  using the universal formulas and the commutative ring structure $(S(T),.,\star)$ extends to $(\cA(A),.,\star)$.
\end{corollary}
\begin{proof}
The condition of associativity and commutativity are polynomial universal relations with integer coefficients that remain true.
\end{proof}

\begin{lemma}[Exponential form of the e\~ne product.]
We have
{\small
$$
f(T)=\prod_{k=1}^{+\infty} (1-X_k T) =\exp\left (- \sum_{n=1}^{+\infty} \frac1n N_n (f) T^n \right ) \ \  \text{with} \ \
N_n (f)=\sum_{k=1}^{+\infty} X_k^n
$$
}
and $N_n(f\star g)= N_n(f).N_n(g)$.

\end{lemma}

\begin{proof}
Develop $-\log (1-X_kT)$ in the exponential and regroup terms for the first statement.
For the last one we use
{\small
$$
\left (\sum_{k=1}^{+\infty} X_k^n \right ).\left (\sum_{l=1}^{+\infty} Y_l^n \right )
= \left (\sum_{k, l=1}^{+\infty} (X_kY_l)^n \right ) \ .
$$
}
\end{proof}
\begin{lemma}[Ghost exponential form of the Big Witt product.]
We have
{\small
$$
f(T)=\prod_{k=1}^{+\infty} (1-X_k T^k)^{-1} =\exp\left (\sum_{n=1}^{+\infty} \frac1n W_n(f) T^n \right ) \ \  \text{with} \ \ W_n(f) = \sum_{d|n} d X_d^{n/d} \ .
$$
}
and  $W_n(f\star g)= -W_n(f).W_n(g)$.
\end{lemma}

\begin{proof}
Same as before for the first statement. Use the previous Lemma for the second.
\end{proof}

The Ghost components $(W_n)$ are used  to define the
classical Big Witt product by
$$
W_n(f\star_w g) =W_n(f).W_n(g) \ .
$$
Therefore,  we have constructed $\star_w$ from the e\~ne product:

\begin{theorem}
We have $f\star_w g = (f\star g)^{-1}$.
\end{theorem}

\begin{corollary}
The result holds for an arbitrary commutative ring $A$.
\end{corollary}
\begin{proof}
We have universal polynomials $(Q_n^w)$ with
{\small
$Q_n^w\in \ZZ[Z_1,\ldots , Z_n]$ such that
$$
f\star_w g = 1+\sum_{k\geq 1} Q_k^w \, T^k = \left ( 1+\sum_{k\geq 1} Q_k \, T^k\right )^{-1} =(f\star g)^{-1}\ .
$$
}
When the monomials are variables, we can replace $(X_k)_{ k\geq 1}$ and $(Y_k)_{ k\geq 1}$ with arbitrary
elements of the ring $A$, in particular in non-zero characteristic since the polynomials have integer coefficients. The identity remains true for all commutative rings.
\end{proof}

\subsection{Historical origin  of the Big Witt ring.}

The original motivation of Witt was the study of cyclic
field extensions of degree a power of a prime number
$p^n$, which was part of the problems of interest to
Hasse's Number Theory school around Class Field theory (see 
the historical survey by Roquette \cite{Ro} and the biography of E. Witt by 
Kersten \cite{Ke}).
In this context, Witt introduced ring structures and Witt
polynomials associated to a prime number $p$, $(W_{p^n})_{n\geq 1}$
in his work from 1937 \cite{Witt}.

Only many years later the full sequence of Bergman-Witt
polynomials $(W_n)_{n\geq 1}$ and the
Big Witt ring appear, almost simultaneously in different places. First, in 1965
in the first edition of Lang's algebra book \cite{La} as an exercise in
section VIII which gives credit  to Witt for an oral communication. In Lang's exercise, the formula for the Bergman-Witt polynomial appears first in print. In Witt's Collected
papers \cite{Witt2} there is an uncirculated manuscript (dated June 23rd 1964) by Witt from a seminar he gave in Hamburg with the construction of the Big Witt ring, but not the explicit
formula of the Bergman-Witt polynomial.
Then, G.M. Bergman,  in  chapter 26 of Mumford's book published in 1966 \cite{Ber}, gives
a full construction of the Big Witt ring and the Bergman-Witt polynomials
appear in page 180. Bergman was a graduate student at the time.
Later, P. Cartier  in  1967 \cite{Ca} gives a very economical construction
of the Big Witt ring based on Bergman-Witt polynomials (although he leaves all details to the
reader). Cartier cites Bergman indirectly by citing
Mumford's book. Cartier's construction runs along the same lines as the one in Lang's book.
He observes that, for each $n\geq 1$, the map from the product ring $W(A)=A^{\NN^*}$ into $A$
$$
W_n(\mathbf{a}) = \sum_{d|n} d a_d^{n/d}
$$
is a ring morphism.
Also he observes that the map
$\mathbf{E} : W(A) \to \cA(A)$ given by
$$
\mathbf{E} (\mathbf{a}) = \prod_{n\geq 1} (1-a_nT^n)^{-1} =f(T)
$$
is a bijection, that satisfes $ \mathbf{E} (\mathbf{a} + \mathbf{b})= \mathbf{E} (\mathbf{a} ).\mathbf{E} (\mathbf{b} )$
and he defines
$$
f\star_w g =\mathbf{E} ( \mathbf{E}^{-1}(f) . \mathbf{E}^{-1}(g)) \ .
$$
So the Witt multiplication is just the multiplication of ghost components.

To add more confusion, years before, in 1958,
A. Grothendieck, in his work in the theory of Chern classes  \cite{Gro},
introduces the notion of $\lambda$-ring structure which have a Big Witt ring structure.
Grothendieck refers to  explicit universal formulas but he does not provide any,
nor the Bergman-Witt polynomials. Also, in a letter to Mumford years later (31st August 1964 \cite{Mu}), 
he praises Bergman ``I liked also Bergman's Chapter 26–27 \cite{Ber}, and especially his universal Witt scheme, realized as a formal power series
functor'' and ``...since Gabriel's seminar on
formal groups I had the feeling that the Witt rings must also have a $\lambda$-ring structure''
(he is talking here about the classical Witt rings from 1937 Witt's article).
Cartier seems to have been the first one to clarify the relation between the Big Witt ring structure and the $\lambda$-ring structure \cite{Ca}.

Since then, the Big Witt structure has appeared in different
branches of mathematics, many times it went unnoticed by a  collective hallucination.
We have seen the example of Manin and Kurokawa tensor product.
Another apparently not well known example is the relation to the theory of symmetric
functions (see the work of A. Lascoux \cite{Las}). \footnote{Another example is the relation
to the e\~ne product, that the second author experienced first hand in a curious
episode. Around the year
2010, the second author spend one morning at the IHES with  Pierre Cartier explaining to him
the e\~ne product and its analytic properties. The presentation was similar to the one given
in \cite{PM}, but over the field $\CC$ and stressing the analytic properties on finite order
meromorphic functions. But Cartier didn't realize the link with the Big Witt ring!
Sometimes the analytic context hides the
algebraic essence of the subject and conversely.}

The diverse appearance in different contexts of the Big Witt ring is
a clear sign of its universal and rich structure.
For other examples and a rich background information, the reader is invited to go through Hazewinkel's survey of this vast subject  \cite{Ha}.

\textbf{Acknowledgements.} We are grateful to Hendrik Lenstra for numerous corrections and the important contribution in Section 3. He also pointed out to us that Lang's Exercise 
in his Algebra book was printed in 1965, so this sets the first appearance in print of the Bergman-Witt polynomials. We are grateful to George M. Bergman for his kind 
reading, corrections and wise advice. 
We thank Ina Kersten for kindly sharing the reference for her biography of Ernst Witt.


\end{document}